\documentclass{amsart}
\usepackage{amssymb,stmaryrd,faktor}

\newtheorem{thm}{Theorem}[section]
\newtheorem*{thma}{Theorem~A}
\newtheorem*{thmb}{Theorem~B}

\newtheorem{cor}[thm]{Corollary}
\newtheorem{claim}{Claim}[thm]
\newtheorem{prop}[thm]{Proposition}

\theoremstyle{definition}
\newtheorem{defn}[thm]{Definition}
\newtheorem{question}[thm]{Question}

\theoremstyle{remark}
\newtheorem{remark}[thm]{Remark}

\DeclareMathOperator{\ch}{\sf CH}

\DeclareMathOperator{\gch}{\sf GCH}
\DeclareMathOperator{\pr}{Pr}
\DeclareMathOperator{\im}{Im}
\DeclareMathOperator{\cf}{cf}
\DeclareMathOperator{\Tr}{Tr}
\DeclareMathOperator{\otp}{otp}

\DeclareMathOperator{\acc}{acc}
\newcommand{\stick}{{\ensuremath \mspace{2mu}\mid\mspace{-12mu} {\raise0.6em\hbox{$\bullet$}}}}
\newcommand{\s}{\subseteq}
\newcommand{\br}{\blacktriangleright}
\newcommand{\symdiff}{\mathbin\triangle}
\newcommand{\prz}[5]{\pr_0(#1,\allowbreak\faktor{{\scriptstyle{#2\circledast#1}}}{ {}^{#3\circledast#1}},\allowbreak#4,#5)}
\newcommand{\pro}[5]{\pr_1(#1,\allowbreak\faktor{{\scriptstyle{#2\circledast#1}}}{ {}^{#3\circledast#1}},\allowbreak#4,#5)}
\newcommand{\pri}[5]{\pr_i(#1,\allowbreak\faktor{{\scriptstyle{#2\circledast#1}}}{ {}^{#3\circledast#1}},\allowbreak#4,#5)}
\renewcommand\mid{\mathrel{|}\allowbreak}

\title[Ramsey theory over partitions II]{Ramsey theory over partitions II:\\Negative Ramsey relations and pump-up theorems}

\author[M. Kojman]{Menachem Kojman}
\address{Department of Mathematics, Ben-Gurion University of the Negev, P.O.B. 653, Be’er Sheva, 84105 Israel}
\urladdr{https://www.math.bgu.ac.il/~kojman/}
\thanks{Kojman was partially supported by the Israel Science Foundation (grant agreement 665/20).}

\author[A. Rinot]{Assaf Rinot}
\address{Department of Mathematics, Bar-Ilan University, Ramat-Gan 5290002, Israel.}
\urladdr{http://www.assafrinot.com}
\thanks{Rinot was partially supported by the Israel Science Foundation (grant agreement 2066/18) and by the European Research Council (grant agreement ERC-2018-StG 802756).}

\author[J. Stepr\={a}ns]{Juris Stepr\={a}ns}
\address{Department of Mathematics \& Statistics, York University, 4700 Keele Street, Toronto, Ontario, Canada M3J 1P3}
\urladdr{http://www.math.yorku.ca/~steprans/}
\thanks{Stepr\={a}ns was partially supported by NSERC of Canada.}

\date{Preprint as of April 29, 2022. For the latest version, visit \textsf{http://assafrinot.com/paper/50}.}

\begin{document}
\begin{abstract} 
In this series of papers we advance Ramsey theory of colorings over partitions.
In this part, we concentrate on anti-Ramsey relations, or, as they
are better known, strong colorings,
and in particular solve two problems from   \cite{strongcoloringpaper}.

It is shown that for every infinite cardinal $\lambda$, a strong coloring on $\lambda^+$ by
$\lambda$ colors over a partition can be stretched to one with $\lambda^{+}$ colors over
the same partition. Also, a sufficient condition is given for when a strong
coloring witnessing $\pr_1(\ldots)$ over a partition may be improved to witness
$\pr_0(\ldots)$.

Since the classical theory corresponds to the special case of a partition with just one cell,
the two results generalize pump-up theorems due to Eisworth and Shelah, respectively.
\end{abstract}
\maketitle

\section{Introduction}
\subsection{Strong colorings}\label{sec11}
Shortly after Ramsey \cite{ramsey} proved his
groundbreaking result that every infinite graph contains an infinite clique or
an infinite anti-clique, Sierpi{\'n}ski \cite{MR1556708} defined a graph over
the reals with neither an uncountable clique nor an uncountable anti-clique.
As a graph may be identified with a $2$-coloring,
Sierpi{\'n}ski's counterexample suggested that there was a class
of \emph{strong colorings} waiting to be discovered on the uncountable
cardinals.
A function $c:[\kappa]^{2}\to\theta$ of unordered pairs from a cardinal $\kappa$ by $\theta$
colors is a \emph{strong coloring} iff for every $A\s \kappa$ of full cardinality $\kappa$
the coloring $c$ hits all colors from $\theta$ on the pairs from $A$, that is
$c[[A]^{2}]=\theta$. 
Such colorings, whose existence is asserted by the symbol
$\kappa\nrightarrow[\kappa]^2_\theta$, witness powerful failures of analogs of Ramsey's theorem.

Surveys of the rich theory of strong colorings that was developed
since Sierpi{\'n}ski's time to the present  may be found in the
introductions to \cite{paper18,strongcoloringpaper}. We mention here, therefore,
only the milestones  which are  most relevant to the present
work: the various ways in which  strong colorings can become stronger.

Sierpi{\'n}ski's example in particular verified that $\aleph_{1}\nrightarrow[\aleph_{1}]^{2}_{2}$
holds. Improving it to handle a larger number of colors was very challenging.
After a few decades and considerable effort by many, Todor{\v{c}}evi{\'c}
extended in \cite{TodActa} Sierpi{\'n}ski's result to one with the maximal
number of colors, which witnessed $\aleph_1\nrightarrow[\aleph_1]^{2}_{\aleph_1}$. Furthermore,
$\kappa\nrightarrow[\kappa]^{2}_{\kappa}$ holds for every uncountable cardinal $\kappa$ that is the successor
of a regular cardinal. Whether Todor{\v{c}}evi{\'c}'s theorem extends to
successors of singulars is still open.

A second  way of  making a strong coloring stronger was to require that
it attains all possible colors on additional graphs beyond
\emph{squares}, i.e., sets of the form $[A]^2=\{(\alpha,\beta)\in A\times A\mid \alpha<\beta\}$.  
Work by Shelah \cite{Sh:280,Sh:327,Sh:572} and by Moore \cite{Moore} has established that
for every $\kappa$ which is a successor of a regular cardinal carries a coloring $c:[\kappa]^2\rightarrow\kappa$ with the
property that $c[A\circledast B]=\kappa$ for all  $A,B\s
\kappa$ of full cardinality, where $A\circledast B$ stands for the \emph{rectangle}
$\{(\alpha,\beta)\in A\times B\mid \alpha<\beta\}$.  We denote this by
$\kappa\nrightarrow[\kappa\circledast\kappa]^2_{\kappa}$.

Assuming the Generalized Continuum Hypothesis ($\gch$), Erd\H{o}s, Hajnal and Rado
\cite[Theorem~17A]{MR202613} constructed for every infinite cardinal $\lambda$ a
coloring $c:[\lambda^+]^2\rightarrow\lambda^+$ with the property that $c[A\circledast B]=\lambda^+$ for all $A\subseteq \lambda^+$
of (small) size $\lambda$ and $B\s \lambda^+$ of size $\lambda^+$, and then Erd\H{o}s, Hajnal and Milner
\cite[Lemma~14.1]{EHM} used $\gch$ to construct a coloring $c:[\lambda^+]^2\rightarrow\lambda^+$ with
the property that for all $A\s\lambda^+$ of size $\lambda$ and $B\s \lambda^+$ of size $\lambda^+$,
there is $\alpha\in A$ such that $c[\{\alpha\}\circledast B]=\lambda^+$. We denote the former by
$\lambda^+\nrightarrow[\lambda\circledast\lambda^+]^2_{\lambda^+}$, and the latter by
$\lambda^+\nrightarrow[\faktor{{\scriptstyle{{\lambda}\circledast\lambda^+}}}{ {}^{1\circledast\lambda^+}}]^2_{\lambda^+}$. Here, the class
of graphs is enlarged and yet all colors are attained on a subgraph of
a prescribed form.

A third aspect in which some strong coloring were shown to be stronger than
others is a coloring's ability to handle patterns of higher dimension. In
\cite{galvin}, Galvin constructed from the Continuum Hypothesis ($\ch$) a
coloring $c:[\aleph_1]^2\rightarrow2$ with the property that for every finite dimension $k$,
every uncountable pairwise disjoint subfamily $\mathcal A\s[\aleph_1]^k$, and every
color $\gamma<2$, there are $a,b\in \mathcal A$ with $\max(a)<\min(b)$ such that
$c[a\times b]=\{\gamma\}$. In Shelah's notation \cite{Sh:282}, this is denoted by
$\pr_1(\aleph_1,\aleph_1,2,{\aleph_0})$. An extension of Galvin's theorem due to
Todor{\v{c}}evi{\'c} is studied in Part~III of this series \cite{paper55}.

Shelah's principle $\pr_1(\kappa,\kappa,\theta,\chi)$ asserts the existence of a coloring $c:[\kappa]^2\rightarrow\theta$ 
with the property that for every $\sigma<\chi$, for every pairwise
  disjoint family $\mathcal A\s[\kappa]^{\sigma}$ of size $\kappa$,  
and every color $\gamma<\theta$,
there are $a,b\in \mathcal A$ with
$\sup(a)<\min(b)$ such that $c[a\times b]=\{\gamma\}$. 
So, $\pr_1(\kappa,\kappa,\theta,2)$ coincides with $\kappa\nrightarrow[\kappa]^2_\theta$,
and $\pr_1(\kappa,\kappa,\theta,3)$ implies the rectangular relation $\kappa\nrightarrow[\kappa\circledast\kappa]^2_{\theta}$.
A survey of key results in the study of $\pr_1(\kappa,\kappa,\theta,\chi)$ is given in the introduction to \cite{paper52}.

Motivated by work of Hajnal and Juh\'{a}sz \cite{MR336705}
and by Roitman \cite{MR486845} that connected strong
colorings and topology, Shelah \cite{Sh:282} identified a fourth
aspect of strengthening a coloring: instead of requiring
$c\restriction(a\times b)$ to be a constant function with some
prescribed value $\gamma$, one requires $c\restriction(a\times b)$ to
realize some arbitrary prescribed finite pattern $g$. 
Specifically, the principle $\pr_0(\kappa,\kappa,\theta,\chi)$ asserts the existence of a coloring $c:[\kappa]^2\rightarrow\theta$
with the property that for every $\sigma<\chi$, for every pairwise
disjoint family $\mathcal A\s[\kappa]^{\sigma}$ of size $\kappa$,  
and every pattern $g:\sigma\times\sigma\rightarrow\theta$,
there are $a,b\in \mathcal A$ with
$\sup(a)<\min(b)$ such that $c(a(i),b(j))=g(i,j)$ for all $i,j<\sigma$.\footnote{Here $a(i)$ stands the for the $i^{\text{th}}$ element of $a$, that is, the unique $\alpha\in a$ to satisfy $\otp(a\cap\alpha)=i$.}

\subsection{Strong colorings over partitions} In order to motivate the definition of strong colorings \emph{over partitions} \cite{strongcoloringpaper},
we first explain how to adapt classical Ramsey relations to this context.
The standard positive Ramsey relation $\kappa\rightarrow(\lambda)^{2}_{\theta}$ 
asserts that for every coloring $c:[\kappa]^2\rightarrow\theta$
there exists a set $A\s\kappa$ of cardinality $\lambda$ such that $c\restriction[A]^2$ is constant.
Such a set  $A$ for which $\{c(\alpha,\beta)\mid(\alpha,\beta)\in [A]^{2}\}$ has size no more than $1$
is called \emph{$c$-homogeneous}.

Given a partition $p:[\kappa]^{2}\rightarrow\mu$ of the unordered pairs from $\kappa$ into $\mu$ cells,
it is possible to relax the notion of $c$-homogeneity to \emph{relative}
$c$-homogeneity \emph{over $p$}. A set $A\subseteq \kappa$ is \emph{$c$-homogeneous over $p$}, or
\emph{$(p,c)$-homogeneous} for short, if all pairs from $A$ which lie in the
same $p$-cell are colored by one color which depends on the cell. More formally,
for every cell $\epsilon<\mu$, the set $\{c(\alpha,\beta)\mid(\alpha,\beta)\in [A]^{2}\ \&\ p(\alpha,\beta)=\epsilon\}$ has size
no more than $1$. Put differently, when restricted to  $[A]^{2}$, $p$ refines $c$, namely the $p$-cell of
a pair in $A$ determines its $c$-color. Thus, $A$ is $(p,c)$-homogeneous iff there
is a function $\tau:\mu\rightarrow\theta$ such that $c(\alpha,\beta)=\tau(p(\alpha,\beta))$ for every $(\alpha,\beta)\in [A]^{2}$.

The Ramsey relation $\kappa\rightarrow(\lambda)^{2}_{\theta}$ can now be relaxed, for a
partition $p$, to its ``over $p$'' version, $\kappa\to_{p}(\lambda)^{2}_{\theta}$, to mean that for
every coloring $c:[\kappa]^{2}\rightarrow\theta$ there is a set $A\subseteq \kappa$ of size $\lambda$ which is
\emph{$c$-homogeneous over $p$}. Note that if $A$ is $c$-homogeneous, it is also
$(p,c)$-homogeneous for every partition $p$ of the pairs from $\kappa$, so
$\kappa\to_{p} (\lambda)^{2}_{\theta}$ follows
from $\kappa\rightarrow(\lambda)^{2}_{\theta}$ for any partition $p$.

Part~I of this series \cite{paper49} is devoted to the consistency of positive Ramsey relations over partitions. 
For instance, it is shown that
Martin's Axiom for $\aleph_{1}$  implies that
$\aleph_{1}\rightarrow_{p}(\aleph_{1})^{2}_{\aleph_0}$ holds for many partitions
$p:[\aleph_1]^{2}\rightarrow\aleph_0$. This means that also the 
Sierpi{\'n}ski coloring can be relatively homogenizied over suitable countable partitions.

In this paper, which constitutes Part~II of this series, we deal with strong colorings over partitions,
that is, with the failure of the ``over p'' version of the weak Ramsey relation
$\kappa\rightarrow[\lambda]^{2}_{\theta}$.

We recall that the relation $\kappa\rightarrow[\lambda]^{2}_{\theta}$
means that for every coloring $c:[\kappa]^2\rightarrow\theta$
there exists a set $A\s\kappa$ of cardinality $\lambda$ such that $c\restriction[A]^2$ is not surjective,
i.e., one of the colors $\gamma<\theta$ is \emph{omitted} 
in the sense that $c(\alpha,\beta)\neq\gamma$ for every pair $(\alpha,\beta)\in[A]^2$.
Now, given a partition $p:[\kappa]^2\rightarrow\mu$,
the relation $\kappa\rightarrow_{p} [\lambda]^{2}_{\theta}$ is defined to mean that for every
coloring $c:[\kappa]^{2}\rightarrow\theta$ there is a set $A\subseteq \kappa$ of cardinality $\lambda$ 
such that at least one color from $\theta$ is omitted by $c$ in every $p$-cell.
That is, for some function $\tau:\mu\rightarrow\theta$,
$c(\alpha,\beta)\not=\tau(p(\alpha,\beta))$ for every $(\alpha,\beta)\in [A]^{2}$.

The strong coloring symbol over $p$, $\kappa\nrightarrow_{p}[\lambda]^{2}_{\theta}$, which is simply the
negation of the positive relation above, means, then, that for every $A\subseteq \kappa$ of
cardinality $\lambda$, for every
function $\tau:\mu\rightarrow\theta$ there is a pair $(\alpha,\beta)\in[A]^2$ such that $c(\alpha,\beta)=\tau(p(\alpha,\beta))$. 
By \cite[Fact~5]{strongcoloringpaper}, this is the same as asserting that 
for every $A\subseteq \kappa$ of cardinality $\lambda$, there is an $\epsilon<\mu$
such that $\{ c(\alpha,\beta)\mid (\alpha,\beta)\in[A]^2\ \&\ p(\alpha,\beta)=\epsilon\}=\theta$.

To sum up, in  Ramsey theory over partitions the role of a color $\gamma$ 
is taken by a function $\tau$ from $p$-cells to colors. Restriction to a single color becomes
restriction  to
a single function, omitting a color becomes omitting a function and  attaining all colors becomes attaining all  functions $\tau$ via $c(\alpha,\beta)=\tau(p(\alpha,\beta))$.

Therefore, the ``over $p$'' versions of the strong coloring principles $\pr_1$ and $\pr_0$ are defined as follows:
\begin{defn}[\cite{strongcoloringpaper}]\label{def11}
Let $p:[\kappa]^2\rightarrow\mu$ be a partition. 
A coloring $c:[\kappa]^2\rightarrow\theta$ is said to witness
\begin{itemize}
\item $\pr_1(\kappa,\kappa,\theta,\chi)_p$ iff 
for every $\sigma<\chi$, every pairwise
  disjoint family $\mathcal A\s[\kappa]^{\sigma}$ of size $\kappa$,
and every function $\tau:\mu\to\theta$ there are $a,b\in \mathcal A$ with $\sup(a)<\min(b)$ such that
$$c(\alpha,\beta)=\tau(p(\alpha,\beta))\text{ for all }\alpha\in a\text{ and }\beta\in b;$$
\item $\pr_0(\kappa,\kappa,\theta,\chi)_p$  iff for
for every $\sigma<\chi$, every pairwise
  disjoint family $\mathcal A\s[\kappa]^{\sigma}$ of size $\kappa$,
and every matrix $(\tau_{i,j})_{ i,j<\sigma}$ of
functions from $\mu$ to $\theta$
there are $a,b\in \mathcal A$ with $\sup(a)<\min(b)$ such that
$$c(a(i),b(j))=\tau_{i,j}(p(a(i),b(j)))\text{ for all }i,j<\sigma.$$
\end{itemize}
\end{defn}

Every coloring which is strong over $p$ is also strong over any partition $p'$ coarser than $p$. However,
no coloring is strong in any of the ways defined above over any non-trivial coarsening of itself (see Proposition~\ref{commonfactor} below).
In particular, for every coloring $c$  by more than one color there are  partitions of pairs to just two cells over which the coloring is no longer strong.

The study of strong colorings over partitions is, then, concentrated on the
strcuture of the space of all colorings witnessing any of the relations defined above over small
partitions: for which colorings $c$ and small partitions $p$, does $c$
witness a negative Ramsey relation over $p$.

\subsection{The results} 
The findings of Part~I show that unlike classical Ramsey theory of the 
uncountable, which steers towards the negative side, 
suitable forcing axioms for a prescribed uncountable cardinal $\lambda$ imply that certain small partitions $p$ satisfy $\lambda^+\rightarrow_p(\lambda^+)^2_\lambda$.
Thus, arbitrary large successor cardinals $\kappa$ may be ``$p$-weakly compact'' in the sense that $\kappa\rightarrow_p(\kappa)^2_2$.

On the other hand, strong negative Ramsey relations over partitions are also
available: Under $\gch$-type assumptions, results from the classical theory
prevail to the new context, for instance, by
\cite[Lemma~9]{strongcoloringpaper}, for any partition $p:[\kappa]^2\rightarrow\mu$ and $i<2$,
$\pr_i(\kappa,\kappa,\theta^\mu,\chi)$ outright implies $\pr_i(\kappa,\kappa,\theta,\chi)_p$.

In the present paper we aim for absolute results rather than independence
results. When the space of strong colorings over a prescribed partition
$p:[\kappa]^2\rightarrow\mu$ is not empty, we explore which colorings this
space must contain.

\medskip

A standard fact from classical theory is that for every successor cardinal $\kappa=\lambda^+$, $\kappa\nrightarrow[\kappa]^2_\lambda$ implies $\kappa\nrightarrow[\kappa]^2_\kappa$.
What happens in the new context? Is it the case that
$\kappa\nrightarrow_p[\kappa]^2_\lambda$ implies $\kappa\nrightarrow_p[\kappa]^2_\kappa$
for every partition $p$? 
And what happens with the stronger relations in Definition~\ref{def11}?
For instance, in \cite{MR3087059}, Eisworth solved a longstanding open problem by
proving that for every singular cardinal $\lambda$,
$\pr_1(\lambda^+,\lambda^+,\lambda,\cf(\lambda))$ implies
$\pr_1(\lambda^+,\lambda^+,\lambda^+,\cf(\lambda))$. 
Is it the case that
$\pr_1(\lambda^+,\lambda^+,\lambda,\cf(\lambda))_p$ implies
$\pr_1(\lambda^+,\lambda^+,\lambda^+,\cf(\lambda))_p$
for every partition $p$?
Whether the same implication holds with  $\lambda=\aleph_0$
was asked in \cite[Question~47]{strongcoloringpaper}.
The first result of this paper answers this question in the affirmative. 

Since the classical theory corresponds to the special case of a partition $p$ with just one cell,
the following result generalizes (and also provides a new proof of) Eisworth’s pump-up theorem:

\begin{thma} 
For every infinite cardinal  $\lambda$, for every partition $p:[\lambda^+]^2\rightarrow\lambda$,
and for every cardinal $\chi\le\cf(\lambda)$,
$$\pr_1(\lambda^+,\lambda^+,\lambda,\chi)_p\iff\pr_1(\lambda^+,\lambda^+,\lambda^+,\chi)_p.$$
\end{thma}

Our proof of the preceding brings the method of walks on ordinals into the study of
strong colorings over partitions.
The proof actually provides an operator from colorings
$c$ by $\lambda$ colors to colorings $c^+$ by $\lambda^+$ colors
with the property that for every partition $p$ to $\le\lambda$ cells,
if $c$ is strong over $p$ then $c^+$ is also strong
over $p$. 

\medskip

In \cite[Lemma~4.5]{sh365}, Shelah presented sufficient cardinal arithmetic conditions for
$\pr_{1}(\kappa,\kappa,\theta,\chi)$ to imply the stronger $\pr_0(\kappa,\kappa,\theta,\chi)$.
\cite[Question~46]{strongcoloringpaper}  asks 
whether it is possible to obtain the same pump-up  over a
partition $p$ for $\kappa=\aleph_1$.  The following theorem
provides a general affirmative answer.

\begin{thmb} For a regular uncountable cardinal $\kappa$ and cardinals $\mu,\lambda,\chi,\theta\le\kappa$
satisfying $\lambda^{<\chi}<\kappa\le2^\lambda$ and
$\lambda^{<\chi}\le\theta^{<\chi}=\theta$,
for every partition
$p:[\kappa]^2\rightarrow\mu$,
$$\pr_1(\kappa,\kappa,\theta,\chi)_p\iff\pr_0(\kappa,\kappa,\theta,\chi)_p.$$
\end{thmb}

\section{Preliminaries}
For cardinals $\chi<\kappa$, $E^\kappa_{\ge\chi}$ denotes the set $\{\alpha<\kappa\mid \cf(\alpha)\ge\chi\}$.
For an ordinal $\sigma$ and a set of ordinals $A$, we write 
$[A]^\sigma$ for $\{ B\s A\mid \otp(B)=\sigma\}$.
For a cardinal $\chi$ and a set $\mathcal A$, we write $[\mathcal A]^\chi:=\{\mathcal B\s\mathcal A\mid |\mathcal B|=\chi\}$ 
and $[\mathcal A]^{<\chi}:=\{\mathcal B\s\mathcal A\mid |\mathcal B|<\chi\}$.
For a set of ordinals $a$ and $b$, we let $\acc(a) := \{\alpha\in a \mid \sup(a \cap \alpha) = \alpha > 0\}$,
and we write $a < b$ if $\alpha < \beta$
for all $\alpha \in a$ and  $\beta \in b$.  
For a set $\mathcal{A}$ which is either an ordinal or a collection of sets
of ordinals, we identify $[\mathcal{A}]^2$ with $\{ (a,b)\in\mathcal A\times\mathcal A\mid a<b\}$.

\begin{defn} For $q:X\rightarrow\nu$ and $p:X\rightarrow\mu$,
$q$ is a \emph{coarsening} of $p$, or $p$ is a \emph{refinement} of $q$, iff
$p(x)=p(y)$ implies  $q(x)=q(y)$ for all $x,y\in X$.
\end{defn}

\begin{prop}\label{commonfactor} For cardinals $\kappa,\nu,\mu,\theta$,
suppose a partition $p:[\kappa]^{2}\rightarrow\mu$ and a coloring $c:[\kappa]^{2}\rightarrow\theta$ have a common coarsening
   $q:[\kappa]^{2}\rightarrow\nu$ that is not constant. Then $c$ does not witness $\kappa\nrightarrow_p[\kappa]^{2}_{\theta}$.
\end{prop}

\begin{proof}
  Let $X$ be an arbitrary $p$-cell, and we shall show that $c\restriction X$ is not surjective.
  
  As $p$ refines $q$, we may fix  a $q$-cell $\hat X$ that covers $X$. As $q$ is not constant there is some $q$-cell $\hat Y$ disjoint from $\hat X$.
  Fix $y\in\hat Y$ and let $\gamma:=c(y)$.
  Since $c$ refines $q$, $c^{-1}[\{\gamma\}]\s \hat Y$. As
  $X\s \hat X$ and $\hat X\cap \hat Y=\emptyset$, it holds that
  $\gamma\neq c(x)$ for all $x\in X$.
\end{proof}

Here and also in Part~III of this series, we study \emph{unbalanced} versions of the principles of Definition~\ref{def11}.

\begin{defn}\label{prelations}
Let $p:[\kappa]^2\rightarrow\mu$ be a partition. 
A coloring $c:[\kappa]^2\rightarrow\theta$ is said to witness
\begin{itemize}
\item $\pro{\kappa}{\nu}{\nu'}{\theta}{\chi}_p$
iff for every pairwise disjoint subfamilies $\mathcal A,\mathcal B$ of $[\kappa]^\sigma$ 
with $|\mathcal A|=\nu$, $|\mathcal B|=\kappa$ and $\sigma<\chi$
there is $\mathcal A'\in[\mathcal A]^{\nu'}$ such that
for every function $\tau:\mu\to\theta$,
there are $a\in\mathcal A'$ and $b\in\mathcal B$ with $a<b$ such that
$$c(\alpha,\beta)=\tau(p(\alpha,\beta))\text{ for all }\alpha\in a\text{ and }\beta\in b;$$

\item $\prz{\kappa}{\nu}{\nu'}{\theta}{\chi}_p$
iff for every pairwise disjoint subfamilies $\mathcal A,\mathcal B$ of $[\kappa]^\sigma$ 
with $|\mathcal A|=\nu$, $|\mathcal B|=\kappa$ and $\sigma<\chi$,
there is $\mathcal A'\in[\mathcal A]^{\nu'}$ such that
for every matrix $(\tau_{i,j})_{ i,j<\sigma}$ of
functions from $\mu$ to $\theta$,
there are $a\in\mathcal A'$ and $b\in\mathcal B$ with $a<b$ such that
$$c(a(i),b(j))=\tau_{i,j}(p(a(i),b(j)))\text{ for all }i,j<\sigma.$$
\end{itemize}
\end{defn}
\begin{remark}

  We write $\pr_i(\kappa,\nu\circledast\kappa,\theta,\chi)_p$ for
$\pri{\kappa}{\nu}{\nu}{\theta}{\chi}_p$. 
\end{remark}

\begin{defn} A partition $p:[\kappa]^2\rightarrow\mu$ is said to have \emph{injective fibers} iff for all $\alpha<\alpha'<\beta$, $p(\alpha,\beta)\neq p(\alpha',\beta)$.
\end{defn}

The next proposition shows that in order to obtain strong colorings over all partitions, it suffices to focus on partitions with injective fibers.

\begin{prop} \label{onlyinjmatters} For every infinite cardinal $\lambda$ and every partition $p:[\lambda^+]^2\rightarrow\lambda$,
  there exists a corresponding partition
  $\bar p:[\lambda^+]^2\rightarrow\lambda$ with injective fibers such that if
any  strong coloring relation from Definitions \ref{def11} and \ref{prelations} holds for $\bar p$, then it also holds for $p$.
\end{prop}
\begin{proof} Given $p:[\lambda^+]^2\rightarrow\lambda$, we define $q:[\lambda^+]^2\rightarrow\lambda\times\lambda$ as follows.
Fix an arbitrary nonzero $\beta<\lambda^+$. Fix a bijection $i_\beta:|\beta|\leftrightarrow\beta$. Then, for every $\epsilon<|\beta|$, let
$$q(i_\beta(\epsilon),\beta):=(p(i_\beta(\epsilon),\beta),\otp\{\varepsilon<\epsilon\mid p(i_\beta(\varepsilon),\beta)=p(i_\beta(\epsilon),\beta)\}).$$

It is easy to check that, for all $\alpha<\beta<\lambda^+$:
\begin{itemize}
\item $q(\alpha,\beta)=(p(\alpha,\beta),\zeta)$ for some $\zeta<\lambda$;
\item $q(\alpha',\beta)\neq q(\alpha,\beta)$  for all $\alpha'<\alpha$.
\end{itemize}
Finally, fix a bijection $\pi:\lambda\leftrightarrow\lambda\times\lambda$ and set $\bar p:=\pi^{-1}\circ q$.

Then, to any function $\tau\in {}^\lambda\theta$, we define the corresponding function $\bar\tau\in{}^\lambda\theta$
such that, for all $\eta<\lambda$, if $\pi(\eta)=(\xi,\zeta)$, then $\bar\tau(\eta)=\tau(\xi)$.
\end{proof}

\section{Theorem~A: increasing the number of colors}\label{pumpingup}

This section deals with the problem of pumping up a strong coloring
by $\lambda$  colors into one by $\lambda^+$ colors. 
The special case of a
singular cardinal $\lambda$ with no partition involved is a result that was first
obtained by Eisworth as a corollary to his transformation theorem of
\cite{MR3087059}. While the theory of transformations has advanced considerably
\cite{paper44,paper45}, at the moment it is unclear whether such transformations
can overcome partitions. Thus, the proof given below is different.

When taking partitions into account, as the phrasing of Question~46 from
\cite{strongcoloringpaper} hints, one ought to expect that different stretchings
of the same coloring might be needed for different partitions. It is surprising,
then, that a coloring can be stretched once in a way which uniformly works for
all partitions it is strong over. Indeed, the main corollary of this section
reads as follows.

\begin{cor} 
  Let $\lambda$ be an infinite cardinal.
  
  For every coloring
$c:[\lambda^+]^2\rightarrow\lambda$
there exists a corresponding coloring $c^+:[\lambda^+]^2\rightarrow\lambda^+$
such that for every partition $p:[\lambda^+]^2\rightarrow\lambda$ and every cardinal $\chi\le\cf(\lambda)$:
\begin{enumerate}
\item if $c$ witnesses $\pr_1(\lambda^+,\lambda^+,\lambda,\chi)_p$
then $c^+$ witnesses $\pr_1(\lambda^+,\lambda^+,\lambda^+,\chi)_p$;
\item if $c$ witnesses $\pr_1(\lambda^+,\lambda^+\circledast\lambda^+,\lambda,\chi)_p$
then $c^+$ witnesses $\pr_1(\lambda^+,{\lambda^+\circledast\lambda^+},\allowbreak\lambda^+,\chi)_p$.
\end{enumerate}
\end{cor}
\begin{proof} The proof is split into three cases:

$\br$ If $\lambda$ is regular then the stationary subset of ordinals of
  cofinality $\lambda$ below $\lambda^{+}$ is non-reflecting, so Theorem~\ref{thm65} below
  applies. 
  
$\br$ If $\lambda$ is singular of countable cofinality, then sets of cardinality
  below $\cf(\lambda)$ are finite, so Theorem~\ref{thm64} below applies. 
  
$\br$ If $\lambda$ is a  singular of uncountable cofinality, then appeal to Theorem~\ref{thm66} below.
\end{proof}

So how do one pump up a strong coloring
by $\lambda$  colors into one by $\lambda^+$ colors?
The classical stretching argument (cf.~\cite[p.~277]{TodActa}) employs a sequence of surjections
$\langle e_\beta:\lambda\rightarrow\beta+1\mid \beta<\lambda^+\rangle$.
Specially, defining $c^+(\alpha,\beta):=e_\beta(c(\alpha,\beta))$ stretches a
strong coloring $c:[\lambda^+]^2\rightarrow\lambda$ to a strong coloring $c^+:[\lambda^+]^2\rightarrow\lambda^+$ via a
pigeonhole consideration for stabilizing the stretch: for a prescribed color
$\delta<\lambda^+$ many $e_\beta$ will map the same $i<\lambda$ to $\delta$, and the
original coloring $c$ will indeed produce any possible $i<\lambda$.

The above one-dimensional stretching (that depends only on $\beta$) is incompatible with a two-dimensional
partition. What we do here, then, is instead of letting $c^+(\alpha,\beta):=e_\beta(c(\alpha,\beta))$,
we let $c^+(\alpha,\beta):=e_\gamma(i)$, where $i$ is again computed
from $c(\alpha,\beta)$, but $\gamma$ is computed from the \emph{triple} $(\alpha,\beta,c(\alpha,\beta))$.

Based on a feedback from the referee, we commence by illustrating the basic idea through a warm-up proof of the simplest pump
up theorem, though this is
covered by later theorems.

\begin{prop}\label{warmup} Let $\lambda$ be an infinite cardinal.

For every coloring $c:[\lambda^+]^2\rightarrow\lambda$,
there exists a corresponding coloring $c^+:[\lambda^+]^2\rightarrow\lambda^+$
such that for every $\mu\le\lambda$ and every $p:[\lambda^+]^2\rightarrow\mu$,
\begin{enumerate}
\item if $c$ witnesses $\lambda^+\nrightarrow_p[\lambda^+]^2_\lambda$
then $c^+$ witnesses $\lambda^+\nrightarrow_p[\lambda^+]^2_{\lambda^+}$;
\item if $c$ witnesses $\lambda^+\nrightarrow_p[\varkappa\circledast\lambda^+]^2_\lambda$
then $c^+$ witnesses $\lambda^+\nrightarrow_p[\varkappa\circledast \lambda^+]^2_{\lambda^+}$.
\end{enumerate}
\end{prop}
\begin{proof}
  Fix a bijection $\pi:\lambda\rightarrow\lambda\times\lambda$. For every $\beta<\lambda^+$, fix a surjection $e_\beta:\lambda\rightarrow\beta+1$
  and let $e_{\gamma}^{{-1}}$ be a right inverse of $e_{\gamma}$, that is, satisfy that
  $e_{\gamma}(e_{\gamma}^{{-1}}(\delta))=\delta$ for all $\delta\le \gamma$.

  Now, given $c:[\lambda^+]^2\rightarrow\lambda$, define
  $c^+:[\lambda^+]^2\rightarrow\lambda^+$, as follows. For $\alpha<\beta<\lambda^+$, let $(i,j):=\pi(c(\alpha,\beta))$ and then
  let:
$$c^+(\alpha,\beta):=e_{e_{\beta}(j)}(i).$$

Thus, rather than apply the bijection $e_{\beta}$ to $c(\alpha,\beta)$ itself, we split
$c(\alpha,\beta)$ to two terms via $\pi$ and apply $e_{\beta}$ to one of them, to find some
$\gamma$ whose $e_{\gamma}$ is applied to the other. This allows the choice of
the stretching to be done by the coloring $c$, and enables the coloring $c^{+}$ to
overcome every partition $p$ which $c$ overcomes.

\begin{claim} For every cofinal $B\s \lambda^+$ there exists $j<\lambda$ such that
$$\sup\left\{\gamma<\lambda^+\mid \sup\{\beta\in B\mid e_\beta(j)=\gamma\}=\lambda^+\right\}=\lambda^+.$$
\end{claim}
\begin{proof} Let $B\s \lambda^+$ be cofinal. For every $\gamma<\lambda^+$ and
  $\beta\ge\gamma$ in $B$ there is some $j_{{\gamma,\beta}}<\lambda$ such that $e_{\beta}(j_{{\gamma,\beta}})=\gamma$.
  As $\lambda^+$ is regular, there is some $j_{\gamma}$ and $B_{\gamma}\s B$ with
  $\sup(B_{\gamma})=\lambda^+$ such that $e_{\beta}(i_{\gamma})=\gamma$ for all $\beta\in B_{\gamma}$. Finally, by
  regularity of $\lambda^+$, there is some $j<\lambda$ such that $j=j_{\gamma}$ for an
  unbounded set of $\gamma<\lambda^+$, as required.
\end{proof}

We shall only verify Clause~(1), and encourage the reader to see they know how to adapt the verification to the context of Clause~(2).

Suppose that $p:[\lambda^+]^2\rightarrow\mu$ with $\mu\le\lambda$ is some partition for which
$c$ witnesses $\lambda^+\nrightarrow_p[\lambda^+]_{\lambda}$. To see that $c^{+}$ witnesses
$\lambda^+\nrightarrow_p[\lambda^+]^2_{\lambda^+}$ suppose that $B\s \lambda^+$ is cofinal and that some
function $\tau:\mu\to\lambda^+$ is given. We need to find a pair $\alpha<\beta$ of ordinal in $B$ such that
$c^{+}(\alpha,\beta)=\tau(p(\alpha,\beta))$.

Let $j<\lambda$ be given by the claim with respect to $B$.
For every $\gamma<\lambda^+$, denote $B_\gamma:=\{\beta\in B\mid e_\beta(j)=\gamma\}$.
By the choice of $j$, the set $\Gamma=\{\gamma<\lambda^+\mid \sup(B_{\gamma})=\lambda^+\}$ is cofinal in $\lambda^+$.
So, by the regularity of $\lambda^+$, we may fix some $\gamma\in \Gamma$ above $\sup (\im(\tau))$.

        Define a function
        $\tau^{*}:\mu \rightarrow\lambda$ as follows. For every $\epsilon<\mu$ let
          $$\tau^{*}(\epsilon)=\pi^{-1}(e_{\gamma}^{{-1}}(\tau(\epsilon)),j).$$

    As $c$ witnesses $\lambda^+\nrightarrow_p[\lambda^+]^{2}_{\lambda}$ and $B_{\gamma}$ is in particular a cofinal subset of $\lambda^+$,
    we may fix $(\alpha,\beta)\in [B_{\gamma}]^2$ such that $c(\alpha,\beta)=\tau^{*}(p(\alpha,\beta))$. Denote
    $(i',j'):=\pi(c(\alpha,\beta))$ and $\epsilon:=p(\alpha,\beta)$.  By the definition of $\tau^*$, necessarily $j'=j$ and $i'=e_{\gamma}^{{-1}}(\tau(\epsilon))$.
    In particular, $e_\beta(j')=\gamma$.
    By the definition of $c^{+}$, then, 
$$c^{+}(\alpha,\beta)=e_{e_{\beta}(j')}(i')=e_{\gamma}(e_{\gamma}^{-1}(\tau(\epsilon)))=\tau(\epsilon)=\tau(p(\alpha,\beta)),$$ as required.
\end{proof}

To motivate the statement of our next theorem, notice that if
$\lambda^+\nrightarrow_p[\lambda^+]^2_{\lambda^+}$ holds, then so do
$\lambda^+\nrightarrow_p[\lambda^+]^2_{\lambda}$ and
$\lambda^+\nrightarrow[\lambda^+]^2_{\lambda^+}$.  
The theorem
shows that it is possible to combine these two  consequences --- a
witness for $\lambda$ many colors over a partition with a witness for
$\lambda^+$ many colors but not over a partition --- into a single
strong coloring. Note that in the next theorem there is no restriction on the value of $\chi$.

\begin{thm}\label{thm62} Suppose $\nu,\mu\le\lambda$ are cardinals with $\lambda$ infinite, and:
\begin{itemize}
\item $p:[\lambda^+]^2\rightarrow\mu$ is a partition;
\item $\nu=1$ or $\nu=\lambda$. More generally,
  $\cf([\lambda]^{\nu},{\s})\le \lambda$ suffices.
\end{itemize}

If $\pro{\lambda^+}{\lambda}{\nu}{\lambda}{\chi}_p$ and $\pr_1(\lambda^+,\lambda\circledast\lambda^+,\lambda^+,\chi)$ both hold,
then so does $\pro{\lambda^+}{\lambda}{\nu}{\lambda^+}{\chi}_p$.
\end{thm}
\begin{proof} 
Fix a coloring $c:[\lambda^+]^2\rightarrow\lambda$ which witnesses
$\pro{\lambda^+}{\lambda}{\nu}{\lambda}{\chi}_p$ and a coloring
$d:[\lambda^+]^2\rightarrow\lambda^+$ which witnesses
$\pr_1(\lambda^+,\lambda\circledast\lambda^+,\lambda^+,\chi)$.  For
every $\beta<\lambda^+$ fix a surjection
$e_\beta:\lambda\rightarrow\beta+1$.  Fix a bijection
$\pi:\lambda\leftrightarrow\lambda\times\lambda$.  Define a coloring
$c^+:[\lambda^+]^2\rightarrow\lambda^+$, as follows:  For all
$\alpha<\beta<\lambda$, let $c^+(\alpha,\beta):=0$; for
$\alpha<\beta<\lambda^+$ with $\beta\ge\lambda$ denote
$(i,j):=\pi(c(\alpha,\beta))$ and  let:
$$c^+(\alpha,\beta):=e_{d(j,\beta)}(i).$$

To verify that $c^+$ witnesses $\pro{\lambda^+}{\lambda}{\nu}{\lambda^+}{\chi}_p$
fix pairwise disjoint subfamilies $\mathcal A,\mathcal B\s[\lambda^+]^{<\chi}$
with $|\mathcal A|=\lambda$ and $|\mathcal B|=\lambda^+$.
Denote $\mathcal B_j^\gamma:=\{b\in\mathcal B\mid \min(b)\ge\lambda\ \&\ d[\{j\}\times b]=\{\gamma\} \}$.
\begin{claim} 
There exists $j<\lambda$ for which
$\{\gamma<\lambda^+\mid |\mathcal B_j^\gamma|=\lambda^+\}$ is cofinal in $\lambda^+$.
\end{claim}
\begin{proof} Suppose not. Then, for every $j<\lambda$,
$\delta_j:=\sup\{\gamma<\lambda^+\mid |\mathcal B_j^\gamma|=\lambda^+\}$ is $<\lambda^+$.
Consider $\delta:=(\sup_{j<\lambda}\delta_j)+1$.
Then, for every $j<\lambda$, $|\mathcal B_j^\delta|<\lambda^+$.
Consequently, $\mathcal B':=\{ b\in\mathcal B\mid \min(b)\ge\lambda\ \&\ \forall j<\lambda\,(d[\{j\}\times b]\neq \{\delta\})\}$ has size $\lambda^+$.
Appealing to $d$ with $\mathcal A':=[\lambda]^1$ and $\mathcal B'$,
there must exist $a\in\mathcal A'$ and $b\in\mathcal B'$ such that $d[a\times b]=\{\delta\}$.
But $a=\{j\}$ for some $j<\lambda$, contradicting the fact that $b\in\mathcal B'$.
\end{proof}

Let $j<\lambda$ be given by the claim.  By the choice of $c$, for
every $\gamma<\lambda^+$ such that $|\mathcal B_j^\gamma|=\lambda^+$
there exists $\mathcal A^\gamma\in[\mathcal A]^\nu$ such that for
every function $\tau:\mu\rightarrow\lambda$, there are $a\in\mathcal A^\gamma$ 
and $b\in\mathcal B_j^\gamma$ with $a<b$ such that
$c(\alpha,\beta)=\tau(p(\alpha,\beta))$ for all 
$(\alpha,\beta)\in a\times b$.  
As $\{\gamma<\lambda^+\mid |\mathcal B_j^\gamma|=\lambda^+\}$ is cofinal in $\lambda^+$ 
and $\cf([|\mathcal A|]^\nu,{\s})<\lambda^+$, we may find some $\mathcal A'\in[\mathcal A]^\nu$ 
for which $\Gamma:=\{\gamma<\lambda^+\mid |\mathcal B_j^\gamma|=\lambda^+\ \&\ \mathcal A^\gamma\s\mathcal A'\}$ is
cofinal in $\lambda^+$.  We claim that $\mathcal A'$ is as sought.

\begin{claim} Let $\tau:\mu\rightarrow\lambda^+$.
There are $a\in\mathcal A'$ and $b\in\mathcal B$ with $a<b$ such that $c(\alpha,\beta)=\tau(p(\alpha,\beta))$ for all $(\alpha,\beta)\in a\times b$.
\end{claim}
\begin{proof}
As $\mu\le\lambda$, we may fix a large enough $\gamma\in\Gamma$ such that $\im(\tau)\s\gamma$.
For every $\epsilon<\mu$, fix $i_\epsilon<\lambda$ such that $e_\gamma(i_\epsilon)=\tau(\epsilon)$.
Define a function $\tau':\mu\rightarrow\lambda$ via $\tau'(\epsilon):=\pi^{-1}(i_\epsilon,j)$.
Pick $a\in\mathcal A^\gamma$ and $b\in\mathcal B_j^\gamma$ with $a<b$ such that $c(\alpha,\beta)=\tau'(p(\alpha,\beta))$ for all $(\alpha,\beta)\in a\times b$.
Clearly, $a\in\mathcal A'$ and $b\in\mathcal B$.
Set $\epsilon:=p(\alpha,\beta)$.
Then $c(\alpha,\beta)=\tau'(\epsilon)=\pi^{-1}(i_\epsilon,j)$,
so that
$c^+(\alpha,\beta)=e_{d(j,\beta)}(i_\epsilon)=e_{\gamma}(i_\epsilon)=\tau(\epsilon)$.
\end{proof}
This completes the proof.
\end{proof}

Another approach for stretching strong colorings for high-dimensional relations, 
is to try to encode sequences of ordinals in a single value. In the next theorem, this is done by appealing
to the Engelking-Karlowicz theorem, which arranges $\lambda^+$ many patterns in a
matrix with just $\lambda$ rows.

\begin{thm}\label{thm64}
Let $\lambda$ be an infinite cardinal.

For every coloring $c:[\lambda^+]^2\rightarrow\lambda$
there exists a corresponding coloring $c^+:[\lambda^+]^2\rightarrow\lambda^+$
such that for every partition $p:[\lambda^+]^2\rightarrow\mu$ with $\mu\le\lambda$
and every cardinal $\chi$ such that $\lambda^{<\chi}=\lambda$:
\begin{enumerate}
\item if $c$ witnesses $\pr_1(\lambda^+,\lambda^+,\lambda,\chi)_p$
then $c^+$ witnesses $\pr_1(\lambda^+,\lambda^+,\lambda^+,\chi)_p$;
\item if $c$ witnesses $\pr_1(\lambda^+,\lambda^+\circledast\lambda^+,\lambda,\chi)_p$
then $c^+$ witnesses $\pr_1(\lambda^+,\lambda^+\circledast\lambda^+,\allowbreak\lambda^+,\chi)_p$;
\item if $c$ witnesses $\pro{\lambda^+}{\lambda}{\nu}{\lambda}{\chi}_p$ with $\cf([\lambda]^\nu,{\s})\le\lambda$,
then $c^+$ witnesses $\pro{\lambda^+}{\lambda}{\nu}{\lambda^+}{\chi}_p$.
\end{enumerate}
\end{thm}
\begin{proof}
Using the Engelking-Karlowicz theorem, fix a sequence
$\langle h_j\mid j<\lambda\rangle$
of functions from $\lambda^+$ to $\lambda$ with the property that  for every
$a\s\lambda^+$ with $\lambda^{|a|}=\lambda$ and a
function $h:a\rightarrow\lambda$, there exists $j<\lambda$
with $h\s h_j$. Define a function $d:\lambda\times\lambda^+\rightarrow\lambda^+$ via 
$$d(j,\beta):=e_\beta(h_j(\beta)).$$

For each $\mathcal B\s\mathcal P(\lambda^+)$, denote $\mathcal B_j^\gamma:=\{b\in\mathcal B\mid \min(b)\ge\lambda\ \&\ d[\{j\}\times b]=\{\gamma\} \}$.
The following is clear.
\begin{claim} Assuming $\lambda^{<\chi}=\lambda$, for every $\mathcal B\s[\lambda^+]^{<\chi}$ of size $\lambda^+$,
there exists $j<\lambda$ for which
$\{\gamma<\lambda^+\mid |\mathcal B_j^\gamma|=\lambda^+\}$ is cofinal in $\lambda^+$.\qed
\end{claim}

The rest of the proof is now very similar to that of Theorem~\ref{thm62}.
We fix a bijection $\pi:\lambda\leftrightarrow\lambda\times\lambda$
and, for every $\beta<\lambda^+$, we fix a surjection $e_\beta:\lambda\rightarrow\beta+1$.
Given a coloring $c:[\lambda^+]^2\rightarrow\lambda$, we define the corresponding
coloring $c^+:[\lambda^+]^2\rightarrow\lambda^+$ 
by letting $c^+(\alpha,\beta):=0$ for all $\alpha<\beta<\lambda$
and, given $\alpha<\beta<\lambda^+$ with $\beta\ge\lambda$,
we denote $(i,j):=\pi(c(\alpha,\beta))$ and  let:
$$c^+(\alpha,\beta):=e_{d(j,\beta)}(i).$$ 

The verification of the three clauses of this theorem is now similar to the verification in the proof of Theorem~\ref{thm62}.
\end{proof}

The proofs of the next theorem and the one following it employ walks on ordinals
in order to pick the $\gamma$ in the template formula ``$c^+(\alpha,\beta):=e_\gamma(i)$''.

\begin{thm}\label{thm65}   Let $\lambda$ be an infinite cardinal.

Suppose that $\chi\le\cf(\lambda)$ and that $E^{\lambda^+}_{\ge\chi}$ admits a non-reflecting stationary set.
For every coloring $c:[\lambda^+]^2\rightarrow\lambda$
there exists a corresponding coloring $c^+:[\lambda^+]^2\rightarrow\lambda^+$
which satisfies that for every partition $p:[\lambda^+]^2\rightarrow\mu$ with $\mu\le\lambda$:
\begin{enumerate}
\item if $c$ witnesses $\pr_1(\lambda^+,\lambda^+,\lambda,\chi)_p$
then $c^+$ witnesses $\pr_1(\lambda^+,\lambda^+,\lambda^+,\chi)_p$;
\item if $c$ witnesses $\pr_1(\lambda^+,\lambda^+\circledast\lambda^+,\lambda,\chi)_p$
then $c^+$ witnesses $\pr_1(\lambda^+,{\lambda^+\circledast\lambda^+},\allowbreak\lambda^+,\chi)_p$.
\end{enumerate}
\end{thm}
\begin{proof} Fix a bijection $\pi:\lambda\leftrightarrow\lambda\times\lambda$.
For every $\beta<\lambda^+$, fix a surjection $e_\beta:\lambda\rightarrow\beta+1$.
Fix a non-reflecting stationary set $\Gamma\s E^{\lambda^+}_{\ge\chi}$
and a surjection $h:\lambda^+\rightarrow\lambda^+$
with the property  that  $H_\gamma:=\{\alpha\in \Gamma\mid h(\alpha)=\gamma\}$ is stationary for all $\gamma<\lambda^+$.
Fix a sequence $\vec Z=\langle Z_\gamma\mid \gamma<\lambda^+\rangle$ of elements of $[\lambda]^{\cf(\lambda)}$
such that, for all $\gamma<\delta<\lambda^+$, $|Z_\gamma\cap Z_\delta|<\cf(\lambda)$.

Let $\vec C=\langle C_\alpha\mid\alpha<\lambda^+\rangle$ be a sequence such that
$C_\alpha$ is a closed subset of $\alpha$ with $\sup(C_\alpha)=\sup(\alpha)$ and $\acc(C_\alpha)\cap\Gamma=\emptyset$, for every $\alpha<\lambda^+$.
We shall be conducting walks on ordinals along $\vec C$ (see \cite{TodWalks} for a comprehensive treatment).
First, for all $\alpha<\beta<\lambda^+$, define a function $\Tr(\alpha,\beta):\omega\rightarrow\beta+1$, by recursion on $n<\omega$, as follows:
$$\Tr(\alpha,\beta)(n):=\begin{cases}
\beta,&n=0\\
\min(C_{\Tr(\alpha,\beta)(n-1)}\setminus\alpha),&n>0\ \&\ \Tr(\alpha,\beta)(n-1)>\alpha\\
\alpha,&\text{otherwise}
\end{cases}$$

Then, derive a function $\rho_2:[\lambda^+]^2\rightarrow\omega$ via
$$\rho_2(\alpha,\beta):=\min\{n<\omega\mid \Tr(\alpha,\beta)(n)=\alpha\}.$$

Now, given a coloring $c:[\lambda^+]^2\rightarrow\lambda$,
we define a corresponding coloring $c^+:[\lambda^+]^2\rightarrow\lambda^+$, as follows.
For every pair $(\alpha,\beta)\in[\lambda^+]^2$, first let $(i,\zeta):=\pi(c(\alpha,\beta))$;
then, if there exists $n<\omega$ such that $\zeta\in Z_{\Tr(\alpha,\beta)(n)}$, let
$$c^+(\alpha,\beta):=e_{h(\Tr(\alpha,\beta)(n))}(i)$$
for the least such $n$.
Otherwise,  let $c^+(\alpha,\beta):=0$.

Observe that the color $c(\alpha,\beta)$ is again split to two terms, but this time the term $\zeta$ is used as a halting condition in the walk,
and then $h$ translates the halting point $\delta$ into the ordinal $\gamma$ corresponding to the surjection $e_{\gamma}$.

\smallskip

To see that $c^+$ is as sought,
let $p:[\lambda^+]^2\rightarrow\mu$ be an arbitrary partition with $\mu\le\lambda$.
Assume one of the following:
\begin{enumerate}
\item  $c$ witnesses $\pr_1(\lambda^+,\lambda^+,\lambda,\chi)_p$
and we are given a pairwise disjoint subfamily $\mathcal A$ of $[\lambda^+]^{<\chi}$ of size $\lambda^+$,
and a prescribed function $\tau:\mu\rightarrow\lambda^+$;
\item  $c$ witnesses $\pr_1(\lambda^+,\lambda^+\circledast\lambda^+,\lambda,\chi)_p$
and we are given two pairwise disjoint subfamilies $\mathcal A,\mathcal B$ of $[\lambda^+]^{<\chi}$ of size $\lambda^+$,
and a prescribed function $\tau:\mu\rightarrow\lambda^+$.
\end{enumerate}

The proofs from either of the assumptions above are very similar. We
will present them simultaneously, indicating by ``Case~(1)'' and
``Case~(2)'' the different parts.

In case~(2), for every $\alpha<\lambda^+$, pick $a_\alpha\in\mathcal A$ and $b_\alpha\in\mathcal B$ with $\min(x_\alpha)>\alpha$, where $x_\alpha:=a_\alpha\cup b_\alpha$.
In case~(1), for every $\alpha<\lambda^+$, pick $a_\alpha\in\mathcal A$ with $\min(x_\alpha)>\alpha$, where  $x_\alpha:=a_\alpha$.

Let $D$ be some club in $\lambda^+$ such that, for every $\delta\in D$ and $\alpha<\delta$, $\sup(x_\alpha)<\delta$.
This ensures that for every $(\alpha,\delta)\in[D]^2$, $\sup(x_\alpha)<\delta<\min(x_\delta)$,
so that $\langle x_\delta\mid \delta\in D\rangle$ is $<$-increasing.

Set $\gamma:=\sup(\im(\tau))$. As $\mu\le\lambda$, it is the case that $\gamma<\lambda^+$,
and hence $\Delta:=\{\delta\in D\cap \Gamma\mid h(\delta)=\gamma\}$ is stationary.
Next, define two functions $f:\Delta\rightarrow\lambda^+$ and $g:\Delta\rightarrow\lambda$ via:
\begin{itemize}
\item $f(\delta):=\sup\{\sup(C_{\Tr(\delta,\beta)(i)}\cap\delta)\mid \beta\in x_\delta, i<\rho_2(\delta,\beta)\}$ and
\item $g(\delta):=\min(Z_\delta\setminus\bigcup\{ Z_{\Tr(\delta,\beta)(i)}\mid \beta\in x_\delta, i<\rho_2(\delta,\beta)\}$.
\end{itemize}

For all $\delta\in\Delta$, $\beta\in x_\delta$ and $i<\rho_2(\delta,\beta)$, $\acc(C_{\Tr(\delta,\beta)(i)})\cap\Gamma=\emptyset$, so  $\sup(C_{\Tr(\delta,\beta)(i)}\cap\delta)<\delta$.
It thus follows from $|x_\delta|<2\cdot\chi\le\cf(\delta)$ that $f(\delta)<\delta$.
Also, since  $|x_\delta|<\cf(\lambda)$, $g(\delta)$ is well-defined.
Fix $(\xi',\zeta')\in\lambda^+\times\lambda$ for which $\Delta':=\{\delta\in\Delta\mid f(\delta)=\xi'\ \&\ g(\delta)=\zeta'\}$ is stationary.

As the prescribed $\tau$ is a function from $\mu$ to $\gamma+1$, we
may fix, for every $\epsilon<\mu$, an $i_\epsilon<\lambda$ such that
$e_\gamma(i_\epsilon)=\tau(\epsilon)$.  Define a function
$\tau':\mu\rightarrow\lambda$ via
$\tau'(\epsilon):=\pi^{-1}(i_\epsilon,\zeta')$.  As $\Delta'\s D$ and
$|\Delta'|=\lambda^+$, we infer that $\mathcal A':=\{ a_\delta\mid \delta\in\Delta'\}$ is a subfamily of $\mathcal A$ of size
$\lambda^+$.  Likewise, in Case~(2), we also have that $\mathcal B':=\{ b_\delta\mid \delta\in\Delta'\}$ is a subfamily of $\mathcal B$
of size $\lambda^+$.  So, in Case~(1) (resp.~Case~(2)), we may fix
$a,b\in \mathcal A'$ (resp.~$a\in\mathcal A'$ and $b\in\mathcal B'$)
with $a<b$ such that $c(\alpha,\beta)=\tau'(p(\alpha,\beta))$ for all
$\alpha\in a$ and $\beta\in b$,

\begin{claim}\label{claim641} Let $(\alpha,\beta)\in a\times b$. Then $c^+(\alpha,\beta)=\tau(p(\alpha,\beta))$.
\end{claim}
\begin{proof}
Denote $\epsilon:=p(\alpha,\beta)$.  By the definition of $\tau'$, $c(\alpha,\beta)=\tau'(\epsilon)=\pi^{-1}(i_\epsilon,\zeta')$.

By the choice of $b$, let us fix $\delta\in\Delta'$ such that $b\s x_\delta$.
As $\beta\in x_\delta$, $\xi'=f(\delta)<\delta<\beta$. Likewise, since $\alpha\in a\in\mathcal A'$, $\xi'<\alpha$.
Altogether,
$$\max\{\sup(C_{\Tr(\delta,\beta)(i)}\cap\delta)\mid i<\rho_2(\delta,\beta)\}\le f(\delta)=\xi'<\alpha<\delta<\beta.$$
Now, by a standard fact from the theory of walks on ordinals (see \cite[Claim~3.1.2]{paper15}),
$\Tr(\alpha,\beta)(i)=\Tr(\delta,\beta)(i)$ for all $i<\rho_2(\delta,\beta)$,
and $\Tr(\alpha,\beta)(\rho_2(\delta,\beta))=\delta$.
Recalling that $g(\delta)=\zeta'$, this means that
$n:=\rho_2(\delta,\beta)$ is the least integer for which  $\zeta'\in Z_{\Tr(\alpha,\beta)(n)}$.
Therefore, by the definition of $c^+$,
$$c^+(\alpha,\beta)=e_{h(\Tr(\alpha,\beta)(n))}(i_\epsilon)=e_{h(\delta)}(i_\epsilon)=e_\gamma(i_\epsilon)=\tau(\epsilon),$$
as sought.
\end{proof}
This completes the proof.
\end{proof}

Unlike the preceding theorem, the proof of the next does not employ
a surjection $h:\lambda^+\rightarrow\lambda^+$,
since it is still open whether for every singular cardinal $\lambda$ 
there is a $\vec C$-sequence that gives rise to a decomposition of $\lambda^+$ into $\lambda^+$ many walk-wise-large sets.
In the template formula ``$c^+(\alpha,\beta):=e_\gamma(i)$'', instead of letting $\gamma:=h(\delta)$ for some well-chosen $\delta$ in the walk from $\beta$ down to $\alpha$,
we shall let $\gamma$ be the $\xi^{th}$ element of $C_\delta$, for well-chosen $\delta$ in the walk \emph{and} $\xi<\lambda$.

\begin{thm}\label{thm66} Suppose that $\lambda$ is a singular cardinal of uncountable cofinality
and $\chi\le\cf(\lambda)$.
For every coloring $c:[\lambda^+]^2\rightarrow\lambda$,
there exists a corresponding coloring $c^+:[\lambda^+]^2\rightarrow\lambda^+$
satisfying that for every partition $p:[\lambda^+]^2\rightarrow\mu$ with $\mu\le\lambda$:
\begin{enumerate}
\item if $c$ witnesses $\pr_1(\lambda^+,\lambda^+,\lambda,\chi)_p$
then $c^+$ witnesses $\pr_1(\lambda^+,\lambda^+,\lambda^+,\chi)_p$;
\item if $c$ witnesses $\pr_1(\lambda^+,\lambda^+\circledast\lambda^+,\lambda,\chi)_p$
then $c^+$ witnesses $\pr_1(\lambda^+,{\lambda^+\circledast\lambda^+},\allowbreak\lambda^+,\chi)_p$.
\end{enumerate}
\end{thm}
\begin{proof} By the proof of Case~1 of Theorem~4.21 from \cite{paper34}, 
we may fix a $C$-sequence $\vec C=\langle C_\alpha\mid\alpha<\lambda^+\rangle$ 
such that $\otp(C_\alpha)<\lambda$ for all $\alpha<\lambda^+$
and such that the functions $\Tr$ and $\rho_2$ derived from walking along $\vec C$ (as defined in the proof of Theorem~\ref{thm65}) satisfy the following.
\begin{claim}\label{c666}
    Let $\mathcal X$ be a pairwise disjoint subfamily of $[\lambda^+]^{<\cf(\lambda)}$ of size $\lambda^+$.
    Then there exists a stationary set $\Delta\s\lambda^+$, a sequence $\langle x_\gamma\mid \gamma\in\Delta\rangle$,
    and an ordinal $\varepsilon<\lambda^+$, such that,
    for every $\gamma\in\Delta$: 
    \begin{itemize}
      \item $x_\gamma\in\mathcal X$ with $\min(x_\gamma)>\gamma>\varepsilon$;
      \item for all $\alpha\in(\varepsilon,\gamma)$ and $\beta \in x_\gamma$, $\gamma\in\im(\Tr(\alpha,\beta))$;
      \item $\cf(\gamma)>\sup\{\otp(C_{\Tr(\gamma,\beta)(n)})\mid \beta\in x_\gamma, n<\rho_2(\gamma,\beta)\}$.
    \end{itemize}
  \end{claim}
\begin{proof} It suffices to prove that for every club $D\s\lambda^+$ there are $\gamma\in D$, $x_\gamma\in\mathcal X$ and an ordinal $\varepsilon<\gamma$ such that 
the  three bullets above hold. 
Now, given an arbitrary club $D\s\lambda^+$,  Claim~4.21.2 from \cite{paper34}
provides  $\gamma\in D$, $x_\gamma\in\mathcal X$ and an ordinal $\varepsilon<\gamma$ such that 
the first two bullets hold.
The proof of that claim makes it clear that
$\cf(\gamma)>|C|$ where $C:=\bigcup\{C_{\Tr(\gamma,\beta)(n)}\mid \beta\in x_\gamma, n\le\rho_2(\gamma,\beta)\}$,
and goes through even if we  require $\cf(\gamma)>|C|^+$.
In particular, this will give $\cf(\gamma)>\sup\{\otp(C_{\Tr(\gamma,\beta)(n)})\mid \beta\in x_\gamma, n<\rho_2(\gamma,\beta)\}$.
\end{proof}

Fix a bijection $\pi:\lambda\leftrightarrow\lambda\times\lambda$.
For every $\beta<\lambda^+$, fix a surjection $e_\beta:\lambda\rightarrow\beta+1$.
Now, given a coloring $c:[\lambda^+]^2\rightarrow\lambda$,
we define a corresponding coloring $c^+:[\lambda^+]^2\rightarrow\lambda^+$ as follows.
For every pair $(\alpha,\beta)\in[\lambda^+]^2$, first let $(i,\xi):=\pi(c(\alpha,\beta))$,
and then, if there exists $n<\omega$ such that $\otp(C_{\Tr(\alpha,\beta)(n)})>\xi$,  let
$$c^+(\alpha,\beta):=e_{C_{\Tr(\alpha,\beta)(n)}(\xi)}(i)$$
for the least such $n$. Otherwise, just let $c^+(\alpha,\beta):=0$.

So here the color $c(\alpha,\beta)$ is split to two terms. The term $\xi$ is used both as a halting condition in the walk
and as a pointer for a specific element $\gamma=C_\delta(\xi)$ (corresponding to the surjection $e_{\gamma}$) in the ladder $C_\delta$ of the halting point $\delta$.

To see that $c^+$ is as sought, let $p:[\lambda^+]^2\rightarrow\mu$ be an arbitrary partition with $\mu\le\lambda$.
There are two cases to consider:
\begin{enumerate}
\item  Assume $c$ witnesses $\pr_1(\lambda^+,\lambda^+,\lambda,\chi)_p$,
and we are given a pairwise disjoint subfamily $\mathcal A$ of $[\lambda^+]^{<\chi}$ of size $\lambda^+$,
and a prescribed function $\tau:\mu\rightarrow\lambda^+$.

In this case, appeal to Claim~\ref{c666} with $\mathcal X:=\mathcal A$,
and obtain a stationary set $\Delta\s\lambda^+$, a sequence $\langle x_\gamma\mid\gamma\in\Delta\rangle$ and an ordinal $\varepsilon<\lambda^+$.
\item  Assume $c$ witnesses $\pr_1(\lambda^+,\lambda^+\circledast\lambda^+,\lambda,\chi)_p$,
and we are given two pairwise disjoint subfamilies $\mathcal A,\mathcal B$ of $[\lambda^+]^{<\chi}$ of size $\lambda^+$,
and a prescribed function $\tau:\mu\rightarrow\lambda^+$.

In this case, appeal to Claim~\ref{c666} with 
some pairwise disjoint subfamily $\mathcal X$ of $[\lambda^+]^{<\cf(\lambda)}$ of size $\lambda^+$ such that, for every $x\in\mathcal X$,
there are $a\in\mathcal A$ and $b\in\mathcal B$ such that $x=a\uplus b$.
In return, we obtain a stationary set $\Delta\s\lambda^+$, a sequence $\langle x_\gamma\mid\gamma\in\Delta\rangle$ and an ordinal $\varepsilon<\lambda^+$.
Then, for every $\gamma\in\Delta$, fix $a_\gamma\in\mathcal A$ and $b_\gamma\in\mathcal B$ such that $x_\gamma=a_\gamma\uplus b_\gamma$.
\end{enumerate}

Next, let $D$ be some club in $\lambda^+$ such that, for every $\delta\in D$ and $\alpha\in\Delta\cap\delta$, $\sup(x_\alpha)<\delta$.
By shrinking $D$, we may also assume that $\min(D)>\sup(\im(\tau))$. 
For every $\delta\in D\cap\Delta$, let $\xi_\delta$ denote the least ordinal $\xi<\lambda$ with $$\sup\{\otp(C_{\Tr(\delta,\beta)(n)})\mid \beta\in x_\delta, n<\rho_2(\delta,\beta)\}<\xi<\otp(C_\delta)$$ 
such that $C_\delta(\xi)>\sup(\im(\tau))$.
Fix $\xi<\lambda$ and $\gamma<\lambda^+$ for which $\Delta':=\{\delta\in \Delta\cap D\mid \xi_\delta=\xi\ \&\ C_\delta(\xi)=\gamma\}$.
As the prescribed  $\tau$ is a function from $\mu$ to $\gamma$,
for every $\epsilon<\mu$, we may fix $i_\epsilon<\lambda$ such that $e_\gamma(i_\epsilon)=\tau(\epsilon)$.
Define a function $\tau':\mu\rightarrow\lambda$ via $\tau'(\epsilon):=\pi^{-1}(i_\epsilon,\xi)$.
As $\Delta'\s D$ and $|\Delta'|=\lambda^+$, we infer that $\mathcal A':=\{ a_\delta\mid \delta\in\Delta'\}$ is a subfamily of $\mathcal A$ of size $\lambda^+$.
Likewise, in Case~(2), we also have that $\mathcal B':=\{ b_\delta\mid \delta\in\Delta'\}$ is a subfamily of $\mathcal B$ of size $\lambda^+$.
So, in Case~(1) (resp.~Case~(2)), we may fix $a,b\in \mathcal A'$ 
(resp.~$a\in\mathcal A'$ and $b\in\mathcal B'$)  with $a<b$
such that
$c(\alpha,\beta)=\tau'(p(\alpha,\beta))$ for all $\alpha\in a$ and $\beta\in b$,

\begin{claim} Let $(\alpha,\beta)\in a\times b$. Then $c^+(\alpha,\beta)=\tau(p(\alpha,\beta))$.
\end{claim}
\begin{proof}
Denote $\epsilon:=p(\alpha,\beta)$.  By the definition of $\tau'$, $\pi(c(\alpha,\beta))=\pi(\tau'(\epsilon))=(i_\epsilon,\xi)$.

By the choice of $b$, let us fix $\delta\in\Delta'$ such that $b\s x_\delta$.
As $\beta\in x_\delta$ and $\alpha\in a\in\mathcal A'$, 
$$\varepsilon<\alpha<\delta<\beta,$$
so that $\delta\in\im(\Tr(\alpha,\beta))$. 
Now, by the same standard fact used in the proof of Claim~\ref{claim641},
$\Tr(\alpha,\beta)(i)=\Tr(\delta,\beta)(i)$ for all $i<\rho_2(\delta,\beta)$,
and $\Tr(\alpha,\beta)(\rho_2(\delta,\beta))=\delta$.
Recalling the choice of $\xi$, this means that
$n:=\rho_2(\delta,\beta)$ is the least integer to satisfy 
$\otp(C_{\Tr(\alpha,\beta)(n)})>\xi$.
So, by the definition of $c^+$, we infer that
$$c^+(\alpha,\beta)=e_{C_{\Tr(\alpha,\beta)(n)}(\xi)}(i)=e_{C_\delta(\xi)}(i_\epsilon)=e_\gamma(i_\epsilon)=\tau(\epsilon),$$
as sought.
\end{proof}
This completes the proof.
\end{proof}

\section{Theorem~B: Strengthening hi-dimensional colorings}
In this short section, we prove Theorem~B.
Its proof follows Shelah's proof of the pump-up from   $\pr_1$ to $\pr_0$ \cite[Lemma~4.5]{sh365} and adds to it considerations to handle the partition.
This answers \cite[Question~46]{strongcoloringpaper} in the affirmative.

\begin{thm} \label{pr0topr1} Suppose that $\kappa$ is a regular uncountable cardinal and $\mu,\lambda,\chi,\theta$ are cardinals $\le\kappa$.
  Assume $\lambda^{<\chi}<\kappa\le2^\lambda$
and $\lambda^{<\chi}\le\theta^{<\chi}=\theta$.  

For every coloring
$c_1:[\kappa]^2\rightarrow\theta$, there exists a corresponding
coloring $c_0:[\kappa]^2\rightarrow\theta$ satisfying that for every partition
$p:[\kappa]^2\rightarrow\mu$:
\begin{enumerate}
\item if $c_1$ witnesses
$\pr_1(\kappa,\kappa,\theta,\chi)_p$, then $c_0$ witnesses
$\pr_0(\kappa,\kappa,\theta,\chi)_p$;
\item if $c_1$ witnesses
$\pro{\kappa}{\nu}{1}{\theta}{\chi}_p$, then $c_0$ witnesses
$\prz{\kappa}{\nu}{1}{\theta}{\chi}_p$.
\end{enumerate}
\end{thm}
\begin{proof} 
As $\kappa\le 2^\lambda$, we may fix an
injective sequence $\langle X_\alpha\mid \alpha<\kappa\rangle$ of subsets of $\lambda$.
\begin{claim}\label{stabwindow} For every $\sigma<\chi$ and
$a\in[\kappa]^{\sigma}$,
there are $y\in[\lambda]^{<\chi}$ and an injection $f:\sigma\rightarrow\mathcal P(y)$,
such that, for all $\alpha\in a$,
$X_{\alpha}\cap y=f(\otp(\alpha\cap a))$.
\end{claim}
\begin{proof} For all $\alpha<\beta<\kappa$, let
$\delta_{\alpha,\beta}:=\min(X_\alpha\symdiff X_\beta)$.
Now, let $y:=\{\delta(\alpha,\beta)\mid \alpha,\beta\in a, \alpha\neq\beta\}$
and then define a function $f:\sigma\rightarrow\mathcal P(y)$ via
$f(i):=X_{a(i)}\cap y$.
Evidently, $y$ and $f$ are as required.
\end{proof}

Consider the following set:
$$W:=\left\{ (y^0,y^1,\mathcal Z,g)\mid y^0,y^1\in[\lambda]^{<\chi},\ \mathcal Z\in[\mathcal P(y^0\cup y^1)]^{<\chi}\text{ and }  g:{\mathcal Z\times \mathcal Z}\rightarrow\theta\right\}.$$
It is clear that $|W|=\theta$, so let us fix an enumeration
$\langle (y^0_j,y^1_j,\mathcal Z_j,g_j)\mid j<\theta\rangle$ of
$W$.

For each $j<\theta$, define a function $h_j:[\kappa]^2\rightarrow\theta$ via:
$$h_j(\alpha,\beta):=\begin{cases}
g_j(X_\alpha\cap y^0_j,X_\beta\cap y^1_j)&\text{if }X_\alpha\cap
y^0_j\in\mathcal Z_j\text{ and }X_\beta\cap y^1_j\in\mathcal Z_j;\\
0&\text{otherwise}.\end{cases}$$

Finally, given a coloring $c_1:[\kappa]^2\rightarrow\theta$, define the coloring
$c_0:[\kappa]^2\rightarrow\theta$ via
$$c_0(\alpha,\beta):=h_{c_1(\alpha,\beta)}(\alpha,\beta).$$

(1) Suppose that $c_1$ witnesses $\pr_1(\kappa,\kappa,\theta,\chi)_p$.
To see that $c_0$ witnesses $\pr_0(\kappa,\kappa,\theta,\allowbreak\chi)_p$
fix an arbitrary $\sigma<\chi$, a $\kappa$-sized pairwise disjoint family
$\mathcal A \s[\kappa]^\sigma$
and a matrix $(\tau_{\xi,\zeta})_ {\xi,\zeta<\sigma}$ of
functions from $\mu$ to $\theta$. By Claim~\ref{stabwindow} and a pigeonhole
argument, fix a set
$y\in[\lambda]^{<\chi}$ and an injection $f:\sigma\rightarrow\mathcal P(y)$
such that, for all $a\in \mathcal A$ and $\alpha\in a$,
$X_{\alpha}\cap y=f(\otp(\alpha\cap a))$.

Denote $\mathcal Z:=\im(f)$.
For every $\epsilon<\mu$, define a function $g^\epsilon:\mathcal Z\times \mathcal Z\rightarrow\theta$ via
$$g^\epsilon(f(\xi),f(\zeta)):=\tau_{\xi,\zeta}(\epsilon).$$
Now pick $j_\epsilon<\theta$ such that
$(y_{j_\epsilon}^0,y_{j_\epsilon}^1,\mathcal Z_{j_\epsilon},g_{j_\epsilon})=(y,y,\mathcal Z,g^\epsilon)$.
Finally, define a function $\tau^*:\mu\rightarrow\theta$ via
$\tau^*(\epsilon):=j_\epsilon$
and then pick $(a,b)\in[\mathcal A]^2$ such that
$c_1(\alpha,\beta)=\tau^*(p(\alpha,\beta))$ for all $\alpha\in a$ and $\beta\in b$.

\begin{claim} \label{checkc0} Let $\xi,\zeta<\sigma$. Then $c_0(a(\xi),b(\zeta))=\tau_{\xi,\zeta}(p(a(\xi),b(\zeta))$.
\end{claim}
\begin{proof} Write $\epsilon:=p(a(\xi),b(\zeta))$. Then
$$c_1(a(\xi),b(\zeta))=\tau^*(p(a(\xi),b(\zeta)))=\tau^*(\epsilon)=j_\epsilon.$$
Altogether,
$$\begin{aligned}c_0(a(\xi),b(\zeta))=&\ h_{j_\epsilon}(a(\xi),b(\zeta))\\=&\ g_{j_\epsilon}(X_{a(\xi)}\cap y^0_{j_\epsilon},X_{b(\zeta)}\cap y^1_{j_\epsilon})\\=&\ g^\epsilon(f(\xi),f(\zeta))=\tau_{\xi,\zeta}(\epsilon),\end{aligned}$$
as sought.
\end{proof}

(2) Suppose that $c_1$ witnesses $\pro{\kappa}{\nu}{1}{\theta}{\chi}_p$.
To see that $c_0$ witnesses $\prz{\kappa}{\nu}{1}{\theta}{\chi}_p$,
fix an arbitrary $\sigma<\chi$, a $\nu$-sized pairwise disjoint family
$\mathcal A \s[\kappa]^\sigma$
and a $\kappa$-sized pairwise disjoint family $\mathcal B \s[\kappa]^\sigma$.
By Claim~\ref{stabwindow}, we may assume the existence of a set
$y^1\in[\lambda]^{<\chi}$ and an injection $f^1:\sigma\rightarrow\mathcal P(y^1)$,
such that, for all $b\in \mathcal B$ and $\beta\in B$,
$X_{\beta}\cap y^1=f^1(\otp(\beta\cap b))$.
Now, by the hypothesis on $c_1$, fix $a\in\mathcal A$ such that, for every
$\tau:\mu\to\theta$,
there exist $b\in\mathcal B$ with $a<b$ such that, for all
$\alpha\in a$ and $\beta\in b$,
$c_1(\alpha,\beta)=\tau(p(\alpha,\beta))$.

By Claim~\ref{stabwindow}, fix $y^0\in[\lambda]^{<\chi}$ and an injection
$f^0:\sigma\rightarrow\mathcal P(y^0)$
such that, for all $\alpha\in a$, $X_{\alpha}\cap y^0=f^0(\otp(\alpha\cap a))$.
Denote $\mathcal Z:=\im(f^0)\cup\im(f^1)$.

Now, given a matrix $(\tau_{\xi,\zeta})_{\xi,\zeta<\sigma}$ of
functions from $\mu$ to $\theta$, for every $\epsilon<\mu$, pick
any function $g^\epsilon:\mathcal Z\times \mathcal Z\rightarrow\theta$
satisfying that for all $\xi,\zeta<\sigma$:
$$g^\epsilon(f^0(\xi),f^1(\zeta))=\tau_{\xi,\zeta}(\epsilon).$$
Then pick $j_\epsilon<\theta$ such that
$(y_{j_\epsilon}^0,y_{j_\epsilon}^1,\mathcal Z_{j_\epsilon},g_{j_\epsilon})=(y^0,y^1,\mathcal Z,g^\epsilon)$.
Finally, define a function $\tau^*:\mu\rightarrow\theta$ via
$\tau^*(\epsilon):=j_\epsilon$
and then pick $b\in\mathcal B$ with $a<b$ such that
$c_1(\alpha,\beta)=\tau^*(p(\alpha,\beta))$ for all $\alpha\in a$
and $\beta\in b$.
The proof of Claim~\ref{checkc0} makes clear that
$c_0(a(\xi),b(\zeta))=\tau_{\xi,\zeta}(p(a(\xi),b(\zeta)))$ for all $\xi,\zeta<\sigma$.
\end{proof}

\section{Concluding remarks}
The pump up theorems which are presented here furnish the toolbox for working
in the theory of strong colorings over partitions with the basic tools which are
available in the classical theory: stretching the number of colors and
strengthening one high dimensional principle to a stronger one.

Also the spectrum
of methods is expanded, with walks on ordinals making their first appearance here.

Together with the independence results of Parts I and III of this series and 
the absolute results of this part, Ramsey theory over partitions emerges now as a
fruitful branch of Ramsey theory on uncountable cardinals.

Many more questions arise. We mention only a few of them.

\begin{question} Can the restriction ``$\chi\le\cf(\lambda)$'' in Theorem~A be waived?
\end{question}
\begin{question}Is it consistent that some positive Ramsey relation
  holds over a small partition at a successor of a singular cardinal?
\end{question}
For a positive integer $n$
and a partition $p:[\kappa]^n\rightarrow\mu$, define $\kappa\nrightarrow_p[\kappa]^n_\theta$ to assert that 
there exists a coloring $c:[\kappa]^n\rightarrow\theta$ such that for every $A\s\kappa$ of full cardinality 
and every $\tau:\mu\rightarrow\theta$, there exists $x\in[A]^n$ such that $c(x)=\tau(p(x))$.
\begin{question}
Is it  possible to extend Todorcevic's \cite{MR1297180} celebrated 
negative Ramsey relation
$\aleph_2 \nrightarrow [\aleph_1]^3_{\aleph_0}$ to the new context? 
Dually, is it consistent that for some small partition $p$ of $[\aleph_2]^3$,  the positive Ramsey relation $\aleph_2 \rightarrow_p[\aleph_1]^3_{\aleph_0}$ holds?
\end{question}

\begin{question}
	For which integers $n$ and $m$ and cardinals $\kappa,\theta$, does
	$\kappa \rightarrow_p [\kappa]_{\theta}^n$ for some partition $p:[\kappa]^n\rightarrow\mu$
	imply that $\kappa \rightarrow_{p^*} [\kappa]_{\theta}^m$
	for some partition $p^*:[\kappa]^m\rightarrow\mu$?
\end{question}

For a successor cardinal $\kappa=\lambda^+$, we can show that the axiom $\stick(\lambda^+)$ entails a positive implication between all $n,m\ge1$.

\end{document}